\documentclass[12pt]{amsart}
\usepackage{amsrefs,amssymb,amsmath,amssymb,amsthm,verbatim}
\usepackage[dvipsnames,svgnames,x11names,hyperref]{xcolor}
\usepackage[colorlinks=true,linkcolor=Maroon,citecolor=blue,urlcolor=blue,hypertexnames=false,linktocpage]{hyperref}
\usepackage{lipsum}
\usepackage{cleveref}
\usepackage{bookmark}
\usepackage{amsmath,thmtools,mathtools}

\newtheorem{theorem}{Theorem}[section]

\newtheorem{lemma}[theorem]{Lemma}

\newtheorem{corollary}[theorem]{Corollary}
\theoremstyle{definition}

\numberwithin{equation}{section}
\newcommand{\M}{\mathcal{M}}
\newcommand{\Rn}{\mathbb{R}^{n+1}}
\newcommand{\Sn}{\mathbb{S}^n}

\theoremstyle{remark}
\newtheorem{remark}[theorem]{Remark}

\numberwithin{equation}{section}

\makeatletter
\@namedef{subjclassname@2020}{%
  \textup{2020} Mathematics Subject Classification}
\makeatother

\begin{document}

\title[Uniqueness of solutions]
 {Uniqueness of solutions to a class of isotropic curvature problems}
\author[]{Mohammad N. Ivaki, Emanuel Milman}
\dedicatory{}
\begingroup    \renewcommand{\thefootnote}{}    
    \footnotetext{The research leading to these results is part of a project that has received funding from the European Research Council (ERC) under the European Union's Horizon 2020 research and innovation programme (grant agreement No 101001677).}
\endgroup
\begin{abstract}
Employing a local version of the Brunn-Minkowski inequality,
we give a new and simple proof of a result due to Andrews, Choi and Daskalopoulos that the origin-centred balls are the only closed, self-similar solutions of the Gauss curvature flow. Extensions to various non-linearities are obtained, assuming the centroid of the enclosed convex body is at the origin. By applying our method to the Alexandrov-Fenchel inequality, we also show that origin-centred balls are the only solutions to a large class of even Christoffel-Minkowski type problems.
\end{abstract}
\maketitle
\section{Introduction}
The Gauss curvature flow in $\mathbb{R}^3$ was proposed by Firey \cite{Fir74} as a model for the changing shape of smooth, strictly convex stones as they tumble on a beach in an idealized situation. Assuming that the solutions exist and are regular, he showed centrally-symmetric stones become round. Firey conjectured that the resulting shapes would be rounded stones even without the symmetry assumption. The existence and regularity of solutions and convergence to a point were settled later by Chou \cite{Tso85}, and in \cite{And99} Andrews succeeded in proving  Firey's conjecture. One of the key ingredients in Andrews' proof was showing that the difference of the principal curvatures decreases along the flow. The question of whether the asymptotic shape in $\mathbb{R}^{n+1}$, $n\geq 3$, is a sphere remained open until Guan and Ni \cite{GN17} showed that the normalized solution converges to a self-similar solution (i.e. a smooth, closed hypersurface whose support function is positive and equal to its Gauss curvature), and Choi and Daskalopoulos \cite{CD16} proved that these self-similar solutions are, in fact, round. Their argument relied on applying the maximum principle to a peculiar combination of principal curvatures and the position vector. In \cite{Sar22}, Saroglou announced a second approach based on the Steiner symmetrization and extended the uniqueness results of \cites{Cho85,And99,AC12,CD16,BCD17} to a non-homogeneous case. See also \cite{McC18} (the last remark), and  \cite{ACGL20}*{Chap. 15-17} for a more detailed account.  

Our first contribution is a new proof of the following theorem; the uniqueness of closed, self-similar solutions of the Gauss curvature flow (for curves the theorem was proved by Gage \cites{Gag84}; see also \cite{And03}). Throughout the paper, all hypersurfaces $\M^n$ are assumed to be closed smooth hypersurfaces in $\mathbb{R}^{n+1}$ bounding a convex compact set $K$ with strictly positive Gauss curvature $\mathcal{K}$ and having the origin in its interior. We denote the support function of $K$ by $h$.  

\begin{theorem}\label{An+BCD}\cites{Gag84,And99,CD16}
Let $\M^n$ be a smooth, strictly convex hypersurface. 
If $\mathcal{K}=h$, then $\M^n$ is the origin-centred unit sphere.
\end{theorem}
We use a local version of the Brunn-Minkowski inequality applied to the position vector of $\M^n$ to provide a surprisingly short proof of this theorem that differs from all the previous approaches discussed above. Our approach also yields (see Theorem \ref{thm:BCD}) the uniqueness of solutions to $\mathcal{K}^{\alpha}=h$ when $\alpha\in [\frac{1}{n+2},\frac{1}{2}]$ (see also \cites{BCD17} for a different argument). In particular, we give a new proof of the following classical theorem due to J\"{o}rgens, Calabi, and Pogorelov \cites{Jor54, Cal58,Pog72,CY86}. See also \cite{ACGL20}*{Sec. 16.4} for another proof based on the Steiner symmetrization.

\begin{theorem}\label{Calabi}
Let $\M^n$ be a smooth, strictly convex hypersurface. If $\mathcal{K}=h^{n+2}$, then $\M^n$ is an ellipsoid centred at the origin.
\end{theorem}

We shall say that $\M^n = \partial K$ is origin-centred if the centroid of $K$ is at the origin; in particular, an origin-symmetric $\M^n$ is origin-centred. The following theorem, when $\varphi$ is \emph{only} a function of the support function (i.e. $\partial_2 \varphi \equiv 0$) \emph{but} without the origin-centred assumption, was proved in \cite{Sar22}. Otherwise, in this general form, it seems to be new.

\begin{theorem}\label{gen of Saroglou}
Suppose $\varphi: (0,\infty)\times (0,\infty)\to (0,\infty)$ is $C^1$-smooth with $\partial_1\varphi\geq 0, \partial_2\varphi\geq 0$, and at least one of these inequalities is strict. If $\M^n$ be a smooth, strictly convex, origin-centred hypersurface with $\varphi(h,|Dh|)\mathcal{K}=h^{n+2}$, then $\M^n$ is an origin-centred sphere. 
\end{theorem}

An immediate corollary of \autoref{gen of Saroglou} is the following uniqueness which confirms a conjecture in \cite{CHLZ23} on the isotropic Gaussian Minkowski problem in the class of origin-centred convex bodies.

\begin{corollary}
Let $n\geq 2$. Suppose $\M^n$ is a smooth, strictly convex, origin-centred hypersurface such that $c e^{\frac{1}{2}|Dh|^2}\mathcal{K}=1$ for some $c>0$. Then $\M^n$ is an origin-centred sphere.
\end{corollary}

There are several known results about the uniqueness of solutions to the isotropic $L_{p,q}$ Minkowski problem $h^{p-1} |Dh|^{n+1-q}\mathcal{K}=c$ with $c>0$:
\begin{itemize}
\item \cite{HZ18}, uniqueness of solutions for $p>q$;
\item \cite{CHZ19}, uniqueness of origin-symmetric solutions for \[-(n+1)\leq p<q\leq \min \{n+1,n+1+p\} ;\]
\item \cite{CL21}, uniqueness of solutions for $1<p<q\leq n+1$, or $-(n+1)\leq p<q<-1$, or $p=q$ (up-to rescaling);
\item \cite{LW22}, complete classification for $n=1$.
\end{itemize}

Here, as another corollary of \autoref{gen of Saroglou}, we state the following uniqueness result. 
\begin{corollary} Let $n\geq 2$ and assume that $-(n+1)\leq p$ and $q\leq n+1$, with at least one of these being strict. 
	 Suppose $\M^n$ is a smooth, strictly convex, origin-centred hypersurface such that  $h^{p-1} |Dh|^{n+1-q}\mathcal{K}=c$ with $c>0$. Then $\M^n$ is an origin-centred sphere. 
\end{corollary}

Let $\sigma_k$ denote the $k$th elementary symmetric function of principal radii of curvature. It is known that the only smooth, strictly convex solution to the isotropic $L^p$-Christoffel-Minkowski problem $h^{1-p}\sigma_k=1$, $1-k \leq p<1,\, 1\leq k<n$, is the unit sphere; see \cite{Che20}  for the case $1-k<p<1$ and \cite{McC11} for $p=1-k$. Here, applying our method to a local form of the Alexandrov-Fenchel inequality, we extend these previous results to the following general formulation.

\begin{theorem}\label{sigma k soliton}
Let $k<n$. Suppose $\varphi: (0,\infty)\times (0,\infty)\to (0,\infty)$ is a $C^1$-smooth function with $k-1+x\partial_1(\log\varphi)(x,y)\geq 0,\,\partial_2\varphi\geq 0$.  Let $\M^n$ be a smooth, strictly convex, origin-symmetric hypersurface such that  $h\sigma_k=\varphi(h,|Dh|)$. Then $\M^n$ is an origin-centred sphere.
\end{theorem}

\section{Background}
Let $(\Rn,\delta:=\langle\,,\rangle,D)$ be the Euclidean space with its standard inner product and flat connection. 
 $(\mathbb{S}^n,\bar{g},\bar{\nabla})$ denotes the unit sphere equipped with its standard round metric and Levi-Civita connection. Moreover, $\mu$ is the spherical Lebesgue measure of the unit sphere.

Let $K \subset \Rn$ be a smooth, strictly convex body with the origin in its interior. We write $\M:=\partial K$ for the boundary of $K$. The Gauss map of $K$, $\nu: \M\to \Sn$, takes $p \in\M$ to its unique outer unit normal vector.
The support function and Gauss curvature of $\M$ are defined as
\begin{align*}
h(x)=\langle \nu^{-1}(x),x\rangle,\quad 
\frac{1}{\mathcal{K}(x)}=\frac{\det (\bar{\nabla}^2h+\bar{g}h)}{\det(\bar{g})}\Big|_x, \quad\quad x\in \Sn.
\end{align*}
The inverse Gauss map $X = \nu^{-1} : \mathbb{S}^n \rightarrow \mathcal{M}$ is given by
\[ X(x) := Dh(x)=h(x)x+\bar{\nabla}h(x)\quad \forall x\in \mathbb{S}^n.\]
Note that $A[h] := \bar{\nabla}^2h+\bar{g} h = D^2 h|_{T \mathbb{S}^n}$ is positive-definite.  
We set
\[\sigma_n=\frac{1}{\mathcal{K}} ,\quad dV=h\sigma_nd\mu. \] 
The measure $\sigma_n d\mu$ is the surface-area measure of $K$, obtained as the push-forward of $\mathcal{H}^{n}|_{\partial K}$ via the Gauss map; the measure $\frac{1}{n+1}V$ is the cone-volume measure of $K$, whose mass is equal to the volume of $K$. We refer to \cites{Sch14,KM22},\cite{ACGL20}*{Sec.~18.7} for additional background. 

For real symmetric $n\times n$ matrices $M_1,\ldots,M_{n}$, write $Q(M_1,\ldots,M_{n})$ for their mixed discriminant; see \cite{Sch14}*{(2.64),(5.117)}. Let $P_{n}$ be the group of all permutations of $\{1,\ldots,n\}$ and $\varepsilon:P_{n}\to \{-1,1\}$ be defined by $\varepsilon(\sigma)=1$ $(-1)$ if $\sigma$ is even (odd). The mixed discriminant of $f_k\in C^{2}(\Sn), 1\leq k\leq n,$ is a multilinear operator defined as
\begin{align*}
Q\left(A[f_1],\ldots, A[f_{n}]\right)
&=\frac{1}{n!}\sum_{a,b\in P_{n}}\varepsilon(a)\varepsilon(b)\prod_{k=1}^{n}(A[f_k])_{a(k)b(k)},
\end{align*}
where in a local orthonormal frame of $\Sn$ the entries of the matrix $A[f_k]$ are given by $(A[f_k])_{ij}=\bar{\nabla}^2_{i,j}f_k+\bar{g}_{ij}f_k$. We define the mixed volume of $f_k\in C^{2}(\Sn), 1\leq k\leq n+1,$ by
\begin{align*}
V(f_1,\ldots,f_{n+1})=\frac{1}{n+1}\int f_1Q\left(A[f_2],\ldots, A[f_{n+1}]\right)d\mu.
\end{align*}
It is known that $V$ is invariant under the permutation of its arguments. 
We also set
\begin{align*}
V_{k+1}(f_1,\ldots,f_{k+1})=V(f_1,\ldots,f_{k+1},1,\ldots,1),
\end{align*}
where $1$ appears $(n-k)$-times on the right-hand side.  
Due to \cite{Sch14}*{Thm.~7.6.8} (cf. \cite{An97}*{Lem.~8}), for all $f\in C^2(\Sn)$ we have the following local version of (a particular case of) the Alexandrov--Fenchel inequality:
\begin{align} \label{A-F}
V_{k+1}(fh,h,\ldots,h)^2\geq V_{k+1}(fh,fh,h,\ldots,h)V_{k+1}(h,\ldots,h).
\end{align}
Equality holds if and only if for some vector $v \in \mathbb{R}^{n+1}$ and constant $c \in \mathbb{R}$ we have
\[f(x)=\langle \frac{x}{h(x)},v\rangle +c\quad \forall x\in \Sn.\]

Let us put $\tau:=A[h]$ and write $\{\lambda_i\}_{i=1}^n$ for its eigenvalues. Define
\begin{align*}
\sigma_k=\sigma_k(\tau)&=\sum_{1\leq i_1<\cdots<i_k\leq n}\lambda_{i_1}\lambda_{i_2}\cdots\lambda_{i_k},\quad \sigma_k^{ij}=\frac{\partial \sigma_k}{\partial \tau_{ij}},\quad 1\leq k\leq n.
\end{align*}
Note that $\sigma_k^{ij}\tau_{ij}=k\sigma_k$ by Euler's identity, and that $\sigma_{n+1}=0$ and $\sigma_1=\bar{\Delta}h+nh$. We also introduce the measures
\[
dV_k := h \sigma_k d\mu ,
\]
so that $V = V_n$. It is known that
\begin{align*}
c'_{k} V_{k+1}(fh,h,\ldots,h) &= \int f h \sigma_k d\mu = \int f dV_k ,\\
c_{k} V_{k+1}(fh,fh,h,\ldots,h) &= \int f h \sigma_k^{ij}(A[fh])_{ij} d\mu . 
\end{align*}

\section{Spectral estimate}

The following is a spectral interpretation of \eqref{A-F}; the case $k=n$ is the spectral formulation of the Brunn-Minkowski inequality originating in Hilbert's work and studied e.g. in \cites{KM22,Mil21}. 

\begin{lemma}[\cite{An97}*{Lem.~8}, \cite{ACGL20}*{Prop.~18.35}] \label{local Vk ineq}
Let $f\in C^2(\Sn)$ with $\int fh\sigma_kd\mu=0$. Then we have
\begin{align}\label{L1BM}
k\int f^2h\sigma_kd\mu\leq \int h^2\sigma_k^{ij}\partial_if\partial_jfd\mu.
\end{align}
Equality holds if and only if for some vector $v\in \mathbb{R}^{n+1}$ we have
\[f(x)=\langle \frac{x}{h(x)},v \rangle\quad \forall x\in \Sn.\]
\end{lemma}
\begin{proof} Due to $\sigma_k^{ij} \tau_{ij} = k \sigma_k$, we have
\begin{align*}
fh \sigma_{k}^{ij}(\bar{\nabla}^2_{i,j}(fh)+\bar{g}_{ij}fh)=kf^2h\sigma_k+fh^2\sigma_{k}^{ij}\bar{\nabla}^2_{i,j}f+2fh\sigma_{k}^{ij}\partial_if\partial_jh.
\end{align*}
Hence, using $\bar{\nabla}_i\sigma_k^{ij}=0$ (see \cite{ACGL20}*{Lem. 18.30}) and integrating by parts we obtain
\begin{align*}
c_kV_{k+1}(fh,fh,h,\ldots,h)&=\int fh \sigma_{k}^{ij}(\bar{\nabla}^2_{i,j}(fh)+\bar{g}_{ij}fh)d\mu\\
&=k\int f^2h\sigma_k d\mu-\int h^2\sigma_k^{ij}\partial_if\partial_jfd\mu.
\end{align*}
Since $V_{k+1}(fh,h,\ldots,h)=0$ and $V_{k+1}(h,\ldots,h) > 0$, the claim follows from \eqref{A-F}, its equality cases, and the fact that $\int x \sigma_k d\mu = 0$; see \cite{Sch14}*{(5.30)}.
\end{proof}

When $K$ is an ellipsoid centred at the origin, for some symmetric, positive definite matrix $M$, we have
\[\langle Dh(x), v\rangle =\langle\frac{x}{h(x)},M v\rangle \quad \forall v\in \Sn.\]
Moreover, $\int Dh \; dV=0,$ and
hence, for $f=\langle Dh, v\rangle$ we have equality in \eqref{L1BM}. In the next lemma, which is the main new ingredient in this work, we derive an inequality from \eqref{L1BM} by substituting such functions for $f$.

\begin{lemma}[Main Lemma]\label{magic inequality}
Let $X=Dh: \Sn\to \partial K$. Then we have
\begin{align*}
k\int |X|^2 dV_k\leq \int h \left ( \sigma_1-(k+1)\frac{\sigma_{k+1}}{\sigma_k} \right )dV_k +k\frac{|\int X dV_k|^2}{\int dV_k}.
\end{align*}
In particular, for $k=n$ we have
\begin{align}\label{magic ineq statement}
n\int|X|^2dV\leq \int h (\bar{\Delta}h+nh)dV+n\frac{\left|\int X dV\right|^2}{\int dV} .
\end{align}
Equivalently, for $k=n$ there holds
\begin{align}\label{affine ver}
\int \langle hX, \bar{\nabla}\log \frac{h^{n+2}}{\mathcal{K}} \rangle dV = \int (n|\bar\nabla h|_{\bar g}^2 - h \bar \Delta h) dV \leq n\frac{|\int XdV|^2}{\int dV}.
\end{align}
\end{lemma}
\begin{proof}Let $\{E_\ell\}_{\ell=1}^{n+1}$ be an orthonormal basis of $\mathbb{R}^{n+1}$. Suppose $\{e_i\}_{i=1}^n$ is a local orthonormal frame for $\Sn$ that diagonalizes $\tau$, say at $x_0$, and $\tau(e_i,e_i)|_{x_0}=\lambda_i(x_0)$. For $ \ell=1,\ldots,n+1$, define the functions
\begin{align*}
f_{\ell}:\Sn\to \mathbb{R},\quad
f_{\ell}=\langle X,E_\ell\rangle-\frac{\int \langle X,E_\ell\rangle dV_k}{\int dV_k}.
\end{align*}
Note that 
\begin{align*}
\int f_\ell dV_k =0,\quad \ell=1,\ldots,n+1.
\end{align*}
Since $e_iX = \tau(e_i,e_j)e_j$ (cf. \cite{CY76}*{(4.15)}) and hence $\partial_i f_\ell = \lambda_i \langle e_i,E_\ell  \rangle$, we obtain
\begin{align*}
\sigma_k^{ij}\partial_i f_\ell \partial_j f_\ell &=\frac{\partial \sigma_k}{\partial \lambda_i}\lambda_i^2\langle e_i,E_\ell \rangle^2,
\end{align*}
and therefore (cf. \cite{HS99}*{Prop.~2.2})
\begin{align*}
\sum_\ell \sigma_k^{ij}\partial_i f_\ell \partial_j f_\ell =\frac{\partial \sigma_k}{\partial \lambda_i}\lambda_i^2=\sigma_1\sigma_k-(k+1)\sigma_{k+1} .
\end{align*}
In addition,
\begin{align*}
\sum_\ell \int f_\ell ^2 dV_k =\int |X|^2 dV_k  -\frac{|\int X dV_k|^2}{\int dV_k}.
\end{align*}
Applying \autoref{local Vk ineq} to $f_{\ell}$ and summing over $\ell$ we obtain the first inequality, and as a particular case when $k=n$ the second inequality. Statement \eqref{affine ver} follows from \eqref{magic ineq statement} after recalling that $|X|^2 = h^2 + |\bar \nabla h|^2_{\bar g}$, $dV = h\sigma_nd\mu$ and integrating by parts.
\end{proof}

\begin{remark}\label{centroid}
By the divergence theorem, for any vector $w \in \mathbb{R}^{n+1}$,
\begin{align*}
\int \langle X,w\rangle dV & =\int_{\partial K} \langle p,w\rangle\langle p,\nu(p)\rangle \mathcal{H}^n(dp) \\
& =\int_{K} \operatorname{div}_{\Rn}(\langle x,w\rangle x ) dx =(n+2)\int_K \langle x,w\rangle dx.
\end{align*}
Therefore, the vector $\int XdV$ appearing in \eqref{magic ineq statement} and \eqref{affine ver} is a multiple of the centroid of $K$, and is equal to $0$ whenever $\M^n = \partial K$ is origin-centred. 
\end{remark}

\begin{remark}
Let $K$ be a smooth, strictly convex body, and 
$M$ be any $(n+1)\times (n+1)$ matrix. By \cite{LW13}*{Prop.~2.1}), we have
\begin{align*}
\int \langle \bar{\nabla}\log \frac{h^{n+2}}{\mathcal{K}}, \xi_M \rangle dV=0,
\end{align*}
where
\[\xi_M(x)=M x-(x^TM x)x,\quad x\in \mathbb{S}^n.\]
In particular, for an ellipsoid $E$ whose support function is given by $h_E(x)=\sqrt{x^TMx}$ (with $M$ symmetric and positive-definite), we have
\[\xi_M= D \frac{h_E^2}{2}(x)-h_E^2(x)x=\frac{1}{2}\bar{\nabla} h_E^2(x),\quad \int \langle \bar{\nabla}\log \frac{h^{n+2}}{\mathcal{K}},\bar{\nabla} h_E^2 \rangle dV =0. \]
Compare this with \eqref{affine ver}: if $K$ has its centroid at the origin, then
\[\int \langle \bar{\nabla}\log \frac{h^{n+2}}{\mathcal{K}},\bar{\nabla} h^2  \rangle dV \leq 0.\]
\end{remark}

\section{Proofs of main results}

\begin{proof}[Proof of \autoref{An+BCD}]By \autoref{magic inequality}, the identity $|X|^2=h^2+|\bar{\nabla} h|_{\bar{g}}^2$, $dV=d\mu$,  and integration by parts  we find
\begin{align}\label{key uniqueness}
(n+1)\int |\bar{\nabla} h|_{\bar{g}}^2d\mu\leq n \frac{\left|\int Xd\mu\right|^2}{\int d\mu}.
\end{align}

By \autoref{centroid}, if $K$ is origin-centred, then the right-hand of \eqref{key uniqueness} is zero and the proof is completed. In the general case, in order to estimate the right-hand side, note that
\begin{align}\label{key uniqueness 2}
\int Xd\mu&=\int \left ( (h(x)-\frac{\int hd\mu}{\int d\mu})x+\bar{\nabla}h(x) \right ) d\mu(x),\nonumber\\
\left|\int Xd\mu\right|^2&\leq\int d\mu \int \left ( (h-\frac{\int hd\mu}{\int d\mu})^2+ |\bar{\nabla}h|_{\bar{g}}^2 \right ) d\mu.
\end{align}
Inequalities \eqref{key uniqueness} and \eqref{key uniqueness 2} together yield
\begin{align*}
\int|\bar{\nabla} h|_{\bar{g}}^2d\mu\leq n\int (h-\frac{\int hd\mu}{\int d\mu})^2d\mu.
\end{align*}
On the other hand, by the  (sharp) Poincar\'{e} inequality on $\mathbb{S}^n$,
\begin{align}\label{key uniqueness 4}
n\int (h-\frac{\int hd\mu}{\int d\mu})^2d\mu\leq \int|\bar{\nabla} h|_{\bar{g}}^2d\mu.
\end{align}
So \eqref{key uniqueness 4} is an equality,  and hence $h = c + \langle x,v \rangle$ and $\M^n$ must be a sphere. Since $\mathcal{K} = h$, $\M^n$ must be the origin-centred unit sphere. 
\end{proof}

The initial part of the above argument immediately extends to yield a simple proof of the following theorem regarding the uniqueness in the isotropic $L^p$-Minkowski problem, originally established in \cite{BCD17} (where the range $p \in (-1,1)$ was also treated via a separate argument):

\begin{theorem} \label{thm:BCD}
Let $\M^n$ be a smooth, strictly convex hypersurface. If $\mathcal{K}=h^{1-p}$ with $p > -(n+1)$, and either $\M^n$ is origin-centred or $p \leq -1$, then $\M^n$ is an origin-centred ball. 
\end{theorem}
\begin{proof}
When $dV = h^p d\mu$, \autoref{magic inequality} and integration by parts yield
\begin{align} \label{p-inq}
\frac{n+1+p}{n}\int |\bar{\nabla} h|_{\bar{g}}^2dV\leq  \frac{\left|\int XdV\right|^2}{\int dV}.
\end{align}
By \autoref{centroid}, we conclude $h$ is constant when $p > -(n+1)$ and $\M^n$ is origin-centred. In the general case, by the divergence theorem,
\begin{align}\label{integ estim}
\int X dV =\frac{n+1+p}{n}\int h^p\bar{\nabla}h d\mu.
\end{align}
Plugging \eqref{integ estim} into \eqref{p-inq}, we deduce when $p > -(n+1)$:
\[\int \left|\bar{\nabla}h-\frac{\int \bar{\nabla}hdV}{\int dV}\right|_{\delta}^2dV\leq \frac{p+1}{n}\frac{|\int \bar{\nabla}hdV|_{\delta}^2}{\int dV}.\]
Consequently, when in addition $p \leq -1$, we conclude that $h$ is constant.
\end{proof}

\begin{proof}[Proof of \autoref{Calabi}]
In the case $p=-(n+1)$, due to \eqref{integ estim}, we have $\int XdV=0$. Thus we have, in fact, equality in \eqref{magic ineq statement}:
\begin{align*}
n\int|X|^2h^{-(n+1)}d\mu= \int  (\bar{\Delta}h+nh)h^{-n}d\mu.
\end{align*}
By the characterization of the equality cases of \eqref{local Vk ineq}, we deduce that for every $v\in \Sn$, there exists a vector $w\in \Rn$, such that
\[\langle X(x),v\rangle=\langle \frac{x}{h(x)},w\rangle\quad \forall x\in \Sn.\]
Therefore, for some matrix $M$, we have
\[\langle Dh(x),v\rangle=\langle \frac{x}{h(x)},Mv\rangle\quad \forall x,v\in \Sn.\]
That is,
$Dh^2(x)=2M^Tx.$
Thus $h^2$ is a quadratic function, and $K$ is an origin-centred ellipsoid. The proof of \autoref{Calabi} is completed.
\end{proof}

\begin{proof}[Proof of \autoref{gen of Saroglou}]
Since $\int X dV=0$, from \eqref{affine ver} it follows that
\begin{align*}
\int \langle X, \bar{\nabla} \frac{h^{n+2}}{\mathcal{K}} \rangle h^{-n} d\mu = \int \langle hX, \bar{\nabla}\log \frac{h^{n+2}}{\mathcal{K}} \rangle dV\leq 0.
\end{align*}
Note that
\[
|X| \langle \bar{\nabla}|X|,\bar{\nabla}h\rangle = \frac{1}{2} \langle D |D h|^2 , \bar{\nabla}h\rangle =\tau(\bar{\nabla}h,\bar{\nabla}h) \geq c |\bar{\nabla}h|^2 ,
\]
where $c>0$ depends on $\M^n$. Due to $h^{n+2}=\varphi(h,|X|)\mathcal{K}$, $\partial_1\varphi\geq 0$ and $\partial_2\varphi \geq 0$ we have
\begin{align*}
\langle X, \bar{\nabla}\frac{h^{n+2}}{\mathcal{K}}\rangle&=|\bar{\nabla}h|^2\partial_1\varphi+\langle \bar{\nabla}|X|,\bar{\nabla}h\rangle\partial_2\varphi\geq c'|\bar{\nabla}h|^2,
\end{align*}
where $c'>0$ depends on $\M^n$ and on the strictness of at least one of the inequalities $\partial_1\varphi > 0$ or $\partial_2\varphi > 0$.
 Hence, $h$ is constant, and so $\M^n$ is an origin-centred sphere.
\end{proof}
\begin{proof}[Proof of \autoref{sigma k soliton}]
In view of \autoref{magic inequality}, and $\int X dV_k=0,$ 
\begin{align*}
k\int |X|^2\varphi d\mu\leq \int (h\varphi\sigma_1-(k+1)h^2\sigma_{k+1}) d\mu.
\end{align*}
Moreover, we have
\begin{align*}
(k+1)\sigma_{k+1}=\sigma_{k+1}^{ij}\tau_{ij} = \sigma_{k+1}^{ij} (\bar{\nabla}^2_{i,j}h+\bar{g}_{ij}h).
\end{align*}
Using $\bar{\nabla}_i\sigma_{k+1}^{ij}=0$, $\partial_2\varphi\geq0$, and integrating by parts we find
\begin{align*}
k\int (h^2+|\bar{\nabla} h|^2_{\bar g})\varphi d\mu
\leq& \int (nh^2\varphi-(\varphi+h\partial_1 \varphi)|\bar{\nabla} h|^2_{\bar g}) d\mu\\
&-\int h^3\sigma_{k+1}^{ij}\bar{g}_{ij} d\mu+2\int h\sigma_{k+1}^{ij}\partial_ih\partial_jh d\mu.
\end{align*}
Choose a local orthonormal frame for $\Sn$ that diagonalizes $\tau$ at $x_0$, so that $\tau(e_i,e_i)|_{x_0}=\lambda_i(x_0)$.
Now due to the identities (cf. \cite{HS99}*{Prop.~2.2})
\begin{align*}
\forall i \;\; \sigma_{k+1}^{ii}=\sigma_{k}-\lambda_i\sigma_k^{ii},\quad 
\sigma_{k+1}^{ij}\bar{g}_{ij}=(n-k)\sigma_k=(n-k)\frac{\varphi}{h},
\end{align*}
we obtain
\begin{align*}
\int (k-1+h\partial_1(\log\varphi))\varphi|\bar{\nabla} h|^2_{\bar g} d\mu +2\int \varphi\lambda_i\frac{\sigma_k^{ii}}{\sigma_{k}}(\partial_ih)^2 d\mu \leq 0.
\end{align*}
Since $k-1+h\partial_1 \log\varphi \geq 0$ and $\lambda_i\frac{\sigma_k^{ii}}{\sigma_{k}} > 0$ for all $i$, we conclude that $h$ is constant. Hence, $\M^n$ is an origin-centred sphere.
\end{proof}
\bibliographystyle{amsalpha-nobysame}

\vspace{10mm}
\textsc{Institut f\"{u}r Diskrete Mathematik und Geometrie,\\ Technische Universit\"{a}t Wien, Wiedner Hauptstra{\ss}e 8-10,\\ 1040 Wien, Austria,} \email{\href{mailto:mohammad.ivaki@tuwien.ac.at}{mohammad.ivaki@tuwien.ac.at}}
		
\vspace{5mm}

\textsc{Department of Mathematics, Technion, Israel Institute of Technology, Haifa 32000, Israel,}
\email{\href{mailto:emilman@tx.technion.ac.il}{emilman@tx.technion.ac.il}}		
		
\end{document}